\documentclass[11pt]{amsart}	
\usepackage[left=1in, right=1in, top=1in, bottom=1in]{geometry}

\UseRawInputEncoding
\usepackage{expex,marvosym,tabularray,xcolor}
\usepackage{tikz-cd}
\usepackage{tikz}
\usepackage{mathdots}
\usetikzlibrary{decorations.pathreplacing}

\usepackage{listings}
\lstdefinelanguage{Macaulay2}{
  morekeywords={apply, degree, gens, ideal, ring, QQ, x_0, x_1, x_2, x_3, x_4, x_5, x_6, x_7, x_8, x_9, x_10, x_11, x_12, x_13, x_14, x_15, x_16, x_17, x_18, x_19, x_20, x_21, x_22, x_23, x_24, x_25, x_26, x_27, x_28, x_29, x_30, x_31, x_32, x_33, x_34, x_35, x_36, x_37, x_38, x_39, x_40, x_41, x_42, x_43, x_44, x_45, x_46, x_47, x_48},
  basicstyle=\small\ttfamily,
  commentstyle=\color{gray}\itshape,
  keywordstyle=\color{blue},
  numbers=left,
  numberstyle=\tiny\color{gray},
  stepnumber=1,
  frame=single,
  numbersep=5pt,
  breaklines=true,
  breakatwhitespace=true,
  captionpos=b,
  stringstyle=\color{red},
  tabsize=2,
  showspaces=false,
  showtabs=false,
  title=\lstname,
  rulecolor=\color{black},
  showstringspaces=false
}

\usepackage{ytableau}
\usepackage{enumerate}
\usepackage{times} %obsolete, but works for the font and style
\usepackage{tipa}
\usetikzlibrary{babel}
\usepackage{paralist}

\usepackage{lipsum} % this and the following package and the settings beneath them are for maintaining indentation and still having ragged-right alignmnent
\usepackage{ragged2e}
\RaggedRight

\usepackage[bottom]{footmisc} % makes the footnotes stock to the bottom of the page
\usepackage{scrextend} % this package and the following settings are for the footnote formatting
\deffootnote[.5em]{0em}{1em}{\textsuperscript{\thefootnotemark}\,}

\newcounter{savefootnote}   % these settings allow the use of an asterisk as a footnote label (for the author line below the title.
\newcounter{symfootnote}
\newcommand{\symfootnote}[1]{%
   \setcounter{savefootnote}{\value{footnote}}%
   \setcounter{footnote}{\value{symfootnote}}%
   \ifnum\value{footnote}>8\setcounter{footnote}{0}\fi%
   \let\oldthefootnote=\thefootnote%
   \renewcommand{\thefootnote}{\fnsymbol{footnote}}%
   \footnote{#1}%
   \let\thefootnote=\oldthefootnote%
   \setcounter{symfootnote}{\value{footnote}}%
   \setcounter{footnote}{\value{savefootnote}}%
}

\usepackage[utf8]{inputenc}
\usepackage{amssymb}
\usepackage[labelsep=period]{caption}
\usepackage[english]{babel}
\usepackage{amsthm,amsmath}
\usepackage{mathrsfs}
\usepackage{tikz-cd}
\usepackage{verbatim}
\usepackage{enumitem}
\usepackage{hyperref}
\theoremstyle{plain}
\theoremstyle{plain}
\newtheorem{thm}{Theorem}[section]

\theoremstyle{definition}
\newtheorem{defn}[thm]{Definition} % definition numbers are dependent on theorem numbers
\newtheorem{exmp}[thm]{Example}
\newtheorem*{remark}{Remark}
\theoremstyle{corollary}
\newtheorem{cor}[thm]{Corollary}

\theoremstyle{lemma}
\newtheorem{lem}[thm]{Lemma}

\theoremstyle{proposition}
\newtheorem{prop}[thm]{Proposition}

\usepackage{lipsum}
\title{Segre Class of Schemes with Regularly Embedded Components}
\author{Guanxi Li}
\date{\today}
\usepackage{natbib}
 
\bibpunct{[}{]}{,}{n}{}{,}
\begin{document}
\begin{abstract}
We generalize Fulton's Residual Intersection Theorem for the Segre class and express the Segre classes of schemes with regularly embedded components in terms of the Chern classes of the normal bundles to the components and their intersections. More specifically, we provide formulas for the following situations: when the components of the scheme intersect transversely and when the ideal sheaf of the scheme, after the blowup along a component, is the product of the ideal sheaves of the exceptional divisor and the residual scheme.  
\end{abstract}

\maketitle

\section{Introduction}
\hspace{8mm}Let $X$ be a closed subscheme of a scheme $Z$ with an ideal sheaf $\mathscr{F}$. \newline Let $C_XZ:= Spec(\sum_{n=0}^{\infty} \mathscr{F}^n / \mathscr{F}^{n+1})$ denote the normal cone to $X$ in $Z$. Consider its projective completion $P(C_XZ\oplus 1)$. Let $\mathscr{O}(1)$ be the dual to tautological line bundle, and let $q:P(C_{X}Z\oplus 1) \rightarrow X$ be the projection. The Segre class of $X$ is then defined as:
$$s(X,Z) := q_*(\sum_{i\geq 0} c_1(\mathscr{O}(1))^i\cap [P(C_XZ \oplus 1)]) $$
\cite[4.2]{Ful84}.
\newline
The Segre class of a vector bundle $E$ on $X$ is defined as follows:
$$s(E) = c(E)^{-1},$$
where $c(\cdot)$ denotes the Chern class, and we write $s(E) = s_1(E)+s_2(E) +s_3(E)\cdots $, where $s_i(E): A_{*}(X)\rightarrow A_{*-i}(X)$.
In the case when $X$ is regularly embedded in $Z$, $N_XY:= C_XY$ denotes the normal bundle to $X$. The Segre class is then expressed as:
$$s(X,Z) = c(N_XZ)^{-1}\cap [X] = s(N_{X}Z)\cap [X].$$

A problem of interest is that of residual intersection. Suppose $X\hookrightarrow Z$ is regularly embedded and $V\subseteq Z$ is a purely $k$-dimensional subscheme. Additionally, let $W = V\cap X.$ Then the intersection product is
$$X\cdot V = \{c(N_XY)\cap s(W, V)\}_{k-d},$$
where $d= codim_Z(X)$ (\cite{Ful84}, 6.1). Here, the subscript of the bracket is interpreted as taking the dimension $(k-d)$-term of the class.

Furthermore, if $U\subseteq W$ is a closed subscheme, the problem of residual intersection involves expressing $X\cdot V$ as a class on $U$ with a residual class $\mathbb{R}$:
$$X\cdot V = \{c(N_{X}Y)\cap s(U,V)\}_{k-d} + \mathbb{R}.$$

At the level of Segre classes, this problem is equivalent to computing $s(W,Z)$ in terms of the sum of $s(U,Z)$ and some residual terms. 
In the case when $U$ is a Cartier divisor, we have Fulton's residual intersection formula:
\begin{prop}[\cite{Ful84} Proposition 9.2]
\label{fultonresidual}
Let $D\subseteq W\subseteq Z$ be closed embeddings, with $Z$ a $k$-dimension variety, and $D$ a Cartier divisor. Let $R$ be a subscheme in $W$ such that $W = D\cup R$ set theoretically and $\mathscr{W} = \mathscr{D}\cdot \mathscr{R}$, where $\mathscr{W}, \mathscr{D}, \mathscr{R}$ are ideal sheaves of $W,D, R$. Then,
$$s(W,Z)_m = s(D,Z)_m+ \sum_{j=0}^{k-m}\binom{k-m}{j}(-D)^j\cdot s(R,Z)_{m+j}$$
in $A_m(W)$. 
\end{prop}
The following definition and corollary are from \cite{Alu15} and \cite{Alu94}:
\begin{defn}
Let $ s(X,Z)= \sum_{i}\alpha^i$, where $\alpha^i$ is the term of codimension $i$ in $X$ of the Segre class. Then, the $\mathscr{L}$-tensored Segre class of $X$ in $Z$ is
$$s(X,Z)^{\mathscr{L}} = \sum_{i}\frac{\alpha^i}{c(\mathscr{L})^{i+1}}$$
\end{defn}
The expression for the Segre class as a whole is the following:
\begin{cor}
\label{alufresidual}
With the above notation, we have:
$$s(W,Z) = s(D,Z)+s(R,Z)^{\mathscr{O}(D)}$$
\end{cor}

In this paper, we aim to generalize Fulton's Residual Intersection formula beyond the case of Cartier divisors. We study the case when $W$ has irreducible components $W_i$, and we aim to express the Segre class $s(W,Z)$ in terms of $s(W_i, Z)$. We introduce the following notation:

\begin{defn}
Let $Z$ be a scheme. Let $E, F$ be vector bundles on $Z$. We define the residual product of the Segre class of vector bundles as follows: 
$$s(E)\odot s(F):= \sum_{k=0}^{\infty}\sum_{\substack{i+j=k\\i,j \geq 1}} \binom{k}{i}s_{i-rank(E)}(E)s_{j-rank(F)}(F),$$
where $s(E)\odot s(F): A_{*}(X)\rightarrow A_{*}(X)$. Then,  note that $s(E)\odot s(F) = \sum_{k=0}^{\infty}\{s(E)\odot s(F)\}^k,$ where  $\{s(E)\odot s(F)\}^k$ is defined to be the operator from $A_{*}(X)$ to $A_{*-k}(X)$. Then, 
$$\{s(E)\odot s(F)\}^k= \sum_{\substack{i+j=k\\i,j \geq 1}} \binom{k}{i}s_{i-rank(E)}(E)s_{j-rank(F)}(F).$$
\end{defn}
\begin{defn}
Let  $Y_{1}, \dots, Y_n$ be regularly embedded subschemes of $Z$. For all subsets $S\subseteq \{1,2,\dots ,n\}$, let $Y_{S}=\bigcap_{i}Y_i$. We say that they intersect transversely, if 
for all subsets $S\subseteq \{1,2,\dots ,n\}$, $N_{Y_S/Z} = \bigoplus_{i}N_{Y_{i}/Z}|_{Y_S}$. 
\end{defn}
Note that the definition of trnaverse intersection that we will be using in this paper requires neither the subvarieties $Y_i$ nor the scheme $Z$ to be nonsingular. 
\begin{thm}
\label{theorem1}
If  $X,Y$ are regularly embedded in the scheme $Z$ with ideal sheaves $\mathscr{I}, \mathscr{J}$. Furthermore if $X, Y$ intersect transversely in $Z$, and let $W$ be the scheme supported on $X\cup Y$ defined by the ideal sheaf $\mathscr{I}\cdot\mathscr{J}$, then, 
$$s(W,Z)=s(X,Z)+s(Y,Z)-s(N_{X}Z)\odot s(N_{Y}Z)\cap [X\cap Y]$$
in $A_*(W)$. The obvious push forward of inclusion is omitted. 
\end{thm}
Transversality is a very strong condition and Segre classes are usually calculated through blowups. Therefore, in practice, one would like to obtain an analogous statement by replacing the transversality condition with some given properties of the scheme after blowups.
\begin{thm}\label{theorem2}
Let $X,Y$ be regularly embedded in the scheme $Z$, with $X\cap Y$ regularly embedded in both $X$ and $Y$. Let $W$ be supported on $X\cup Y$, and consider the blowup $\pi: Bl_YZ \rightarrow Z.$ If $\pi^{-1}(W)$ is supported on $ E\cup \tilde{X}$, such that the proper transform of  $X$, $\tilde{X}$, is residual to the exceptional divisor $E$. Then,
$$s(W,Z) = s(X,Z)+s(Y,Z)- \frac{c(N_{X\cap Y}X)c(N_{X\cap Y}Y)}{c(N_{X\cap Y}Z)}\cdot (s(N_{X\cap Y}Y)\odot s(N_{X\cap Y}Y))\cap [X\cap Y] $$
in $A_*(W)$. The obvious pushforwards of inclusion is omitted. 
\end{thm}
%Let $i:X\cap Y\hookrightarrow Z$ be the embedding. $X$ is regularly embedded in $Z$. Then,
%$$s(W,Z) = s(X,Z)+s(Y,Z)- c(N_{X\cap Y}X)\cdot \left(s(N_{X\cap Y}Y)\odot s(N_{X\cap Y}Y)\right)\cap i^*s(Y,Z)$$

\vspace{5mm}
\section*{Acknowledgments}
The author thanks Paolo Aluffi for countless insightful conversations. The author is supported by the Caltech Summer Undergraduate Research Fellowship (SURF) program and the Ernest R. Roberts SURF fellowship.
\vspace{5mm}
\section{Proof of  Theorem \ref{theorem1}}

We will first prove a lemma regarding the properties of our notion of transverse intersection. 
\begin{lem}\label{technical}
Let  $X,Y$ be regularly embedded subvarieties of the scheme $Z$, intersects transversely. $X\cap Y$ is regularly embedded in $X$ and $Y$. Let $\pi: Bl_{Y}Z\rightarrow Z$ denote the blowup of $Z$ at $Y$,  $\tilde{Z} = Bl_{Y}Z$ and $\tilde{X}$ be the proper transform of $X$. Then:
\begin{enumerate}[label=(\alph*)]
    \item $N_{\tilde{X}}\tilde{Z}= \pi^*N_{X}Z$.
    \item $\pi^{-1}(X) = \tilde{X}$
\end{enumerate}
\end{lem}
\begin{proof}
To show (a), since we have the normal sheaves $\mathscr{N}_{X}Y = (\mathscr{I}_{X}/\mathscr{I}^2_{X})^{\vee}$ and  $\mathscr{N}_{\tilde{X}}\tilde{Z} = (\mathscr{I}_{\tilde{X}}/\mathscr{I}^2_{\tilde{X}})^{\vee}$
Since $\mathscr{N}_XY$ is locally free, then $\pi^*(\mathscr{N}_XY)\simeq (\pi^*(\mathscr{I}_{X}/\mathscr{I}^2_{X}))^{\vee}$. 
Therefore, it is enough to show that $\pi^*\mathscr{I}_X \simeq \mathscr{I}_{\tilde{X}}.$ Note that this would imply (b) as well.

We will check the above isomorphism at the level of the local ring at the closed point. Note that, if $p \notin X\cap Y$ then the blowup does not affect $X$, and $\pi$ is the identity. Therefore $\pi^*\mathscr{I}_{X,p}\simeq \mathscr{I}_{\tilde{X},\pi^{-1}(p)}$. Hence, it is sufficient to check the closed point $p$ in $X\cap Y$, and we wish to show that for all points $q\in \pi^{-1}(p)$, $(\pi^*\mathscr{I}_X)_q \simeq \mathscr{I}_{\tilde{X},q}$. Let $A$ be $\mathscr{O}_{Z,p}$, where $p$ is a closed point in $X\cap Y$,  Since $X, Y$ are regularly embedded, the stalks of their ideals sheaves over the closed points can be generated by regular sequences. Based on the proof of previous lemma, we can assume that  $I_X:= \mathscr{I}_{X,p}$  is generated by the regular sequence $(a_1,a_2, \cdots, a_n)$ and $I_Y :=  \mathscr{I}_{Y,p}$ is generated by the regular sequence $ (b_{1}, b_2, \cdots b_m)$

By \cite{Ful84} B.6.10, over any closed point $q$ such that $\pi(q) = p$, the structure sheaf of the blowup evaluated at $q$ can be written as the following,
$$\mathscr{O}_{\tilde{Z},q}= A[x_1,x_2, \cdots, x_{m-1}]/(b_1-b_{m}x_1, b_2-b_{m}x_2, \cdots, b_{m-1}-b_{m}x_{m-1})$$
Now consider $\tilde{X}= Bl_{X\cap Y}X$. Since $X\cap Y \hookrightarrow X$ is regular embedding, we have $I_{X\cap Y}\otimes \mathscr{O}_{X,p}:=(\bar{b}_1, \bar{b}_2,\cdots, \bar{b}_m)$ also forms a regular sequence.
$$\mathscr{O}_{\tilde{X},q} = \mathscr{O}_{X,p}[x_1,x_2,\cdots,x_{m-1}]/(\bar{b}_1-\bar{b}_mx_1, \cdots, \bar{b}_{m-1} - \bar{b}_m x_{m-1})$$$$ = A[x_1,x_2,\cdots, x_{m-1}]/(a_1,\cdots, a_n,b_1-b_{m}x_1,  \cdots, b_{m-1}-b_{m}x_{m-1} )$$
Then the inclusion map $\tilde{X}\hookrightarrow \tilde{Z}$ induces surjection $\mathscr{O}_{\tilde{Z},q}\rightarrow \mathscr{O}_{\tilde{X},q}$, given by $x_i\mapsto x_i$. The kernel is exactly $(a_1,a_2,\cdots, a_n)$. This is exactly $(\pi^*\mathscr{I}_X)_q$. Thus, we have the equality $\mathscr{I}_{\tilde{X},q} \simeq (\pi^*\mathscr{I}_{X})_q$.

\end{proof}
\vspace{10mm}
With the above lemma, we can prove \ref{theorem1}:
\begin{proof}
It is equivalent to proving the following:
$$s(W, Z)_d=s(X,Z)_d+s(Y,Z)_d-\sum_{i=1}^{dim(Z)-d-1} \binom{dim(Z)-d}{i} s_{dim(Y)-d-i}(N_YZ)s_{dim(X)+i-dim(Z)}(N_XZ).$$
By Proposition 2.2 we have the following commutative diagram: 
\begin{tikzcd}
Bl_{X\cup Y}Z  \arrow[r, "\pi_Y"] \arrow[dr,"\pi"]  & Bl_{X}Z \arrow[d, "\pi_X"]\\  & Z
\end{tikzcd}
Let $E_X$ be the exceptional divisor of the first blowup and $E_Y$ be the exceptional divisor of the second blowup.
Since the ideal sheaf of $W$ is given by $\mathscr{I}\cdot \mathscr{J}$ , then,$$Bl_{W}Z = Bl_{\tilde{Y}}Bl_XZ = Bl_{\tilde{X}}Bl_YZ.$$ We denote $\pi:=\pi_X \circ \pi_Y$. We denote the proper transform of $Y$under blowing up at $X$ to be $\tilde{Y}$ and the blowup of $Z$ at $X$, $\tilde{Z}$. Let $\tilde{E}_X$ be the proper transform of $E_X$ under $\pi_Y$. Then, $E_{W}:=\pi^{-1}(X\cup Y) = \tilde{E}_X  \cup E_Y$.
By \ref{technical}(b), we have the following two facts:
\begin{enumerate}
    \item $\pi_{Y}^{-1}E_X= \tilde{E}_X$
    \item $\pi_X^{-1}(Y) = \tilde{Y}$, which implies  $\pi^{-1}(Y)=\pi_{Y}^{-1}\pi_{X}^{-1}Y= E_Y$
\end{enumerate}

Then using \cite{Ful84} Proposition 4.2(a) we have: 
$$s(X,Z) = \pi_{X*}s(\pi_X^{-1}(X), \tilde{Z}) = \pi_{X*}s(E_X, \tilde{Z})=\pi_{X*}\bigl(c(\mathscr{O}(E_X))^{-1}\cap [E_X]\bigl) = \pi_{X*}\frac{[E_X]}{1+[E_X]}$$
Furthermore, we have: 
$$s(X,Z) = \sum_{d=0}^{dimX} s(X,Z)_d= \sum_{d=0}^{dimX}s_{dimX-d}(N_{X}Z)\cap [X]$$
and on the right-hand side, we have:
$$\pi_*\frac{[E_X]}{1+[E_X]} = \pi_*E_X-\pi_*E_X^2 + \pi_*E_X^3- \dots$$
Since both sides are graded by dimension, we have $\pi_X^*E_X^{dim(Z)-d} = (-1)^{dim(Z)-d+1}s(X,Z)_{d}.$
%$s(Y,Z)= \pi_*s(\pi_Y^{-1}\pi_X^{-1}(Y), \tilde{Z})$, and since $\pi_X^{-1}(Y) \simeq \tilde{Y}$,
%then we have $s(Y,Z)=\pi_*s(\pi^{-1}_2\tilde{Y}, \tilde{Z})=\pi_*s(E_Y,\tilde{Z})$. Therefore, $\pi_*E_Y^k = (-1)^{k+1}s(Y,Z)_{d}$
Similarly, we also have $\pi_{Y*}E_Y^{dim(Z)-d} = (-1)^{dim(Z)-d+1}s(\tilde{Y},\tilde{Z})_{d}.$ Also,
$$s(W,Z)= \pi_*\frac{[E_{X\cup Y}]}{1+[E_{X\cup Y}]} =  \pi_*\frac{[\tilde{E}_{X}]+[E_{Y}]}{1+[\tilde{E}_{X}]+[E_{Y}]}$$
Then we have:

\begin{align*}
s(W, Z)_d & = (-1)^{dim(Z)-d+1} \pi_*(\tilde{E}_X+E_Y)^{dim(Z)-d} \\
                                & = (-1)^{dim(Z)-d+1}  \sum_{i=0}^{dim(Z)-d} \binom{dim(Z)-d}{i} \pi_*(E_Y^{dim(Z)-d-i} \tilde{E}_X^i) \\
                                & =s(X,Z)_{d}+s(Y,Z)_d+  (-1)^{dim(Z)-d+1}  \sum_{i=1}^{dim(Z)-d-1} \binom{dim(Z)-d}{i} \pi_*(E_Y^{dim(Z)-d-i} \tilde{E}_X^i)
\end{align*}

We want to show that $ \pi_*(E_Y^{dim(Z)-d-i} \tilde{E}_X^i) = \pi_*E_Y^{dim(Z)-d-i}\pi_*\tilde{E}_X^{i}$.
Note that $\tilde{E}_X = \pi_Y^{-1}(E_X)$. By the projection formula:
$\pi_*(E_Y^{dim(Z)-d-i}\tilde{E}_X^i) = \pi_{X*}(\pi_{Y*}(E_Y^{dim(Z)-d-i})E_X^i).$
Since $\tilde{Y}\hookrightarrow \tilde{Z}$ is regularly embedded,
\begin{align*}
\pi_{Y*}E^{dim(Z)-d-i}_Y &= (-1)^{dim(Z)-d-i+1}s(\tilde{Y}, \tilde{Z})_{d+i}\\
&= (-1)^{dim(Z)-d-i+1}\left\{\frac{[\tilde{Y}]}{c(N_{\tilde{Y}}\tilde{Z})}\right\}_{d+i}\\
&= (-1)^{dim(Z)-d-i+1} s_{dimY-d-i}(N_{\tilde{Y}}\tilde{Z})\cap [\tilde{Y}].
\end{align*}
By \ref{technical}, $N_{\tilde{Y}}\tilde{Z} \simeq \pi^*_X N_{Y}Z$. 
Then,
\begin{align*}
\pi_*(E_Y^{dim(Z)-d-i}\tilde{E}_X^i) &=\pi_{X*}\left((-1)^{dim(Z)-d-i+1}\left\{\frac{[\tilde{Y}]}{c(N_{\tilde{Y}}\tilde{Z})}\right\}_{d+i}\cdot E_X^i\right) \\
&=(-1)^{dim(Z)-d-i+1}s_{dim(Y)-d-i}(N_{Y}Z)\cdot \pi_{X*}(E_X^i \cap \tilde{Y}).  
\end{align*}
Note that by the universal property of blowup, $E_X\cap \tilde{Y}$ is the exceptional divisor of the blowup $Bl_{X\cap Y}Y,$ and since $E_X^i \cap \_ : A_*(\tilde{Z})\rightarrow A_{*-i}(E)$. Then, $(E_X^i \cap \_)|_{\tilde{Y}}: A_{*}(\tilde{Y})\rightarrow A_{*-i}(E_X\cap \tilde{Y})$ is the same as $(E\cap \tilde{Y})^i \cap \_ $ . Therefore,
$$\pi_X*(E^i_X\cap \tilde{Y}) = (-1)^{i+1}s(X\cap Y, Y)_{dim(Y)-i}.$$
Hence, we have the above:
\begin{align*}
\pi_*(E_Y^{dim(Z)-d-i}\tilde{E}_X^i) & =(-1)^{dim(Z)-d}s_{dim(Y)-d-i}(N_{Y}Z)\cdot s(X\cap Y, Y)_{dimY-i}\\
&=(-1)^{dim(Z)-d}s_{dim(Y)-d-i}(N_{Y}Z)\cdot s_{dim(X)+i-dim(Z)}(N_{X}Z)\cap [X\cap Y].
\end{align*}
Hence $s(W, Z)_d$ is:
$$s(X,Z)_d+s(Y,Z)_d-\sum_{i=1}^{dim(Z)-d-1} \binom{dim(Z)-d}{i} s_{dim(Y)-d-i}(N_YZ)s_{dim(X)+i-dim(Z)}(N_XZ)\cap[X\cap Y].$$
\end{proof}
\begin{remark}
For the case when $X=D$ is a divisor.
\newline Note that $s(D,Z)_{d} = \{\frac{D}{1+D}\}_d = -(1)^{dim(Z)-d+1}D^{dim(Z)-d}$. Let $W$ be the scheme defined by the product of ideal sheaves of $D$ and $Y$. Then,
$$ s(W,Z)_d =s(D,Z)_d+\sum_{i=0}^{dim(Z)-d-1} \binom{dim(Z)-d}{i} (-D)^is(Y,Z)_{d+i}.$$
This is exactly Fulton's Residual Intersection formula \cite{Ful84} Proposition 9.2.
\end{remark}
We want to generalize our previous formula to when there are several components, all intersect transversely. Here we will see that our $\odot$ notation will simplify the notation quite a bit. We will first prove that this notation is indeed well defined.

\begin{lem}
Let $\alpha_i:= s(E_i)$ and $\alpha _i^j =s_j(E_i) $, then,
$$\{\alpha_1\odot\dots \odot \alpha_n \}^k= \sum_{\substack{i_1+\dots i_n = k \\ i_j\geq 1}}\binom{k}{i_1,\dots,i_n}\alpha_1^{i_1-rank(E_i)}\alpha_2^{i_2-rank(E_2)}\dots \alpha_n^{i_n-rank(E_n)},$$
where $\binom{k}{i_1,\dots , i_n}$ is the multinomial coefficient and  $\alpha_{j(dim(Z)-i_j)}$ denotes the dimension $(dim(Z)-i_j)$-th term of the class $\alpha_j$. 

\end{lem}
Note that for the case $n=2$, this is consistent with our definition because the binomial coefficient $\binom{k}{i}$ is the same as the multinomial coefficient $\binom{k}{i, k-i}$. 
\begin{proof}

Assume the formula holds for dotting $n$ times, then for the $(n+1)$th times:
\begin{align*}\{\alpha_1 \odot \dots \alpha_n \odot \alpha_{n+1}\}^k &= \sum_{\substack{p+i_{n+1}=k\\p,i_{n+1} \geq 1}} \binom{k}{i_{n+1}}\{\alpha_1\odot\dots \odot\alpha_n\}^{p}\alpha_{n+1}^{i_{n+1}-rank(E_{n+1})}\\
= \sum_{\substack{p+i_{n+1}=k\\p,i_{n+1} \geq 1}} &\binom{k}{i_{n+1}}\sum_{\substack{i_1+\dots i_n =p \\ i_j\geq 1}}\binom{p}{i_1,\dots,i_n}\alpha_{1}^{i_1-rank(E_1)}\dots \alpha_{n}^{i_n-rank(E_n)}\alpha_{n+1}^{i_{n+1}-rank(E_{n+1})}
\end{align*}
Note that $\binom{k}{i_{n+1}}\binom{p}{i_1, \dots, i_n}= \frac{k!}{i_{n}!p!}\frac{p!}{i_1!\dots i_{n}!} = \frac{k!}{i_1!\dots i_n!i_{n+1}!}=\binom{k}{i_1,\dots,i_{n+1}}$
Then the equality becomes:
$$=\sum_{\substack{i_1+\dots i_{n+1} =k \\ i_j\geq 1}}\binom{k}{i_1,\dots,i_n}\alpha_{1}^{i_1-rank(E_1)}\dots \alpha_{n}^{i_n-rank(E_n)}\alpha_{n+1}^{i_{n+1}-rank(E_{n+1})}$$
By induction, the formula holds.

%proof of (2)
%This follows immediately from $\frac{1}{c(\pi^*\mathscr{E})}\cap \beta = \frac{1}{c(\mathscr{E})} \pi_* \beta$  for any $\beta\in A^*(Y)$.

\end{proof}

\begin{cor}\label{many}
Let $Y_1, Y_2, \dots, Y_n$ be regularly embedded subschemes of a scheme $Z$ with ideal sheaves $\mathscr{I}_1, \mathscr{I}_2, \dots,  \mathscr{I}_n $. Furthermore, $Y_1, Y_2, \dots, Y_n$ intersect transversely in $Z$. Let $W$ be supported on $\bigcup_{i=1}^{k}Y_i$, defined by the ideal $\mathscr{I}_1\cdot \mathscr{I}_2 \dots \mathscr{I}_n $, then:
$$s(W,Z)_d = \sum_{j=1}^{n}(-1)^{j+1}\sum_{\substack{i_1,\dots, i_j \in N\\ i_p \neq i_q}}\left\{\bigodot_{k=1}^js(N_{Y_{i_k}}Z)\}  \cap [\bigcap_{i_{k}\in \{i_{1},i_{2}\cdots,i_{j}\}} Y_{i_k}]\right\}_{d} $$
\end{cor}
\begin{proof}
Consider the blowup $\pi: Bl_{W}Z \rightarrow  Z.$ Because the ideal sheaf of $W$ is $\mathscr{I}_1\cdot \mathscr{I}_2\dots \mathscr{I}_n$, it is sufficient to fix a blowup sequence, say:
$$Bl_{W}Z=Bl_{\tilde{Y}^{(n-1)}_n}\tilde{Z}^{(n-1)}\xrightarrow{\pi_n}Bl_{\tilde{Y}^{(n-2)}_{n-1}}\tilde{Z}^{(n-2)}\xrightarrow{\pi_{n-1}}\dots Bl_{\tilde{Y_2}^{(1)}}\tilde{Z}^{(1)}\xrightarrow{\pi_1}Bl_{Y_1}{Z},$$
where $\pi = \pi_1\circ\dots \circ\pi_n$. Here we adopt the notation that $\tilde{Z}^{(i)}$ to be the proper transform of $Z$ at the $i$-th blowup, similarly for $\tilde{Y}_i^{(j)}$. 
Then applying  \ref{technical}, we have $\tilde{E}_i :=\pi^{-1}(Y_i) = \tilde{Y}_i^{(n-1)}$ for $i\neq n$ and $E_n =  \pi^{-1}(Y_n)$ are Cartier divisors, and the exceptional divisor of $Bl_{W}Z $ is
$E = \tilde{E}_1+\tilde{E}_2+\dots + E_n$. 
We claim that the pushforward is still the same, $\pi_*(\tilde{E}_i)^k = (-1)^{k+1}s(Y_i,Z)_{dim(Z)-k}.$ 
This is because we have three situations for the pushforward. First,  for all $j> i$, $\tilde{Y}_i^{(j)}$ are Cartier divisors and $\pi_{j*} \tilde{Y}_i^{(j)} = \tilde{Y}_i^{(j-1)}$ by \ref{technical} (b) in the previous proof. Second, For $j=i$, we have $\pi_{j*}\tilde{Y}_{i}^{(j)} = -s(\tilde{Y}_{i}^{(j-1)}, \tilde{Z}^{(j-1)})_{dim(Z)-1} $. Third, for $j < i$, we have $\pi_{j*}[\tilde{Y}_i^{(j)}] = [\tilde{Y}_i^{(j-1)}]$ by fact 1 in the previous proof, then $\pi_{j*}s(\tilde{Y}_{i}^{(j)}, \tilde{Z}^{(j)}) = \pi_{j*}\frac{[\tilde{Y}_{i}^{(j)}]}{c(\pi_j^*N_{\tilde{Y}_{i}^{(j-1)}}\tilde{Z}^{(j-1)})}= \frac{[\tilde{Y}_{i}^{(j-1)}]}{c(N_{\tilde{Y}_{i}^{(j-1)}}\tilde{Z}^{(j-1)})} = s(\tilde{Y}_{i}^{(j-1)}, \tilde{Z}^{(j-1)})$. Together with all three cases, we proved our claim.
Hence applying the Multinomial theorem gives us the following:
$$\sum_{\substack{i_1+\dots i_n = k\\ i_j\geq 0}} \binom{k}{i_1, \dots, i_n}\pi_*(E_1^{i_1}\dots E_n^{i_n}) $$
Note that we want to show the following distributive property: $$\pi_*(\tilde{E}_1^{i_1}\tilde{E}_2^{i_2}\dots E_{n}^{i_n}) = \pi_*(\tilde{E}_1)^{i_1}\pi_*(\tilde{E}_2)^{i_2}\dots \pi_*(E_n)^{i_n}.$$ Note that $\pi$ is given by composition of blowup map $\pi_1 \circ \dots\circ \pi_n$, where $\pi_i$ is the blowup at $\tilde{Y}_i^{(i-1)}$ . 
By induction, the base case is proven in the previous theorem. Assume that the above distributive property holds for product $n=k$,. Then for $n=k+1$, we have:
$$\pi_*(\tilde{E}_1^{i_1}\tilde{E}_2^{i_2}\dots E_{k+1}^{i_{k+1}})= (\pi_{1}\circ \dots \circ \pi_{k})_*((E_1^{'i_1}\dots  E_{k}^{'i_{k}})\pi_{k+1*}E_{k+1}^{i_{k+1}})$$
$$=(\pi_{1}\circ \dots \circ \pi_{k})_*((E_1^{'i_1}\dots  E_{k}^{'i_{n-1}})(-1)^{i_{k+1}+1}s(\tilde{Y}^{(k)}_{k+1}, \tilde{Z}^{(k)})_{dim(Z)-i_n}) $$

where $E_{j}^{'i_j}= \pi_{k+1}^*(\tilde{E}_j^{i_j})$, and note that $ (\pi_{1}\circ \dots \circ \pi_{n-1})^*(N_{Y_{k+1}}Z)= N_{\tilde{Y}^{(k)}_{k+1}} \tilde{Z}^{(k)}$.
Then we have the above is:
$$=(\pi_{1}\circ \dots \circ \pi_{k})_*((E_1^{'i_1}\dots  E_{k}^{'i_{k}})\cap [\tilde{Y}^{(k)}_{k+1}])(-1)^{i_{k+1}+1}s_{i_{k+1}+dimY_{k+1} - dim(Z)}(N_{Y_{k+1}}Z)$$
By the induction hypothesis,
$$= (-1)^{i_1+1}s_{i_{k+1}+dimY_{k+1} - dim(Z)}(N_{Y_{k+1}}Z)\dots (-1)^{i_{k+1}+1}s_{i_{k+1}+dimY_{k+1} - dim(Z)}(N_{Y_{k+1}}Z) \cap [\cap_{a=1}^{k+1}Y_a]$$
$$= \pi_*(\tilde{E}_1)^{i_1}\pi_*(\tilde{E}_2)^{i_2}\dots \pi_*(E_{k+1})^{i_{k+1}}$$
Then, all we need to deal with is the sign of the pushforward. The sign of the pushforward only depends on the number of $\pi_*(E_i)$ appears in the monomial. Let $N = \{1,2,\dots, n\}$, we have the following formula:
$$\sum_{i=1}^n s_{dim(Y_i)-d}(N_{Y_i}Z)\cap [Y_i] - \sum_{\substack{i_1,i_2\in N\\ i_1\neq i_2}}\{s(N_{Y_{i_1}}Z)\odot s(N_{Y_{i_2}}Z)\}^{dim(Y_{i_1}\cap Y_{i_2})-d}\cap[Y_{i_1}\cap Y_{i_2}] + \cdots$$
$$+(-1)^{n+1}\{\bigodot_{i=1}^k s(N_{Y_{i}},Z))\}^{dim(\cap_{j=1}^n Y_{i_j})-d}\cap [\bigcap_{j=1}^n Y_{i_j}].$$
\end{proof}
Now we want to show that our transversality condition indeed implies that the ideal sheaf of the union is the product of the ideal sheaf of each component. 
\begin{lem}\label{product lemma}
Suppose $Y_{1}, \dots, Y_n $ that are reduced subvarieties in a reduced scheme $Z$, intersect transversely in $Z$ , and are regularly embedded in $Z$. Furthermore, for any subsets $S\subseteq \{1,2,\dots ,n\}$, $Y_{S}=\bigcap_{i}Y_i$ is regularly embedded in $Z$ , and let $\mathscr{I}_i$ be the ideal sheaf for $Y_i$. Then the scheme $W$ defined by the ideal sheaf $\mathscr{I}_1 \cdot\mathscr{I}_2\dots \mathscr{I}_n$ is reduced. 
\end{lem}
The motivation behind this lemma is that $X\cup Y$ is defined set-theoretically but not always scheme-theoretically. However, with the condition in the lemma, the union can be defined scheme-theoretically, to be the reduced scheme supported on the set $X\cup Y$,  where we take the scheme given by the product of the ideal sheaves. In this case, and we obtain a more geometric of the statement of \ref{theorem1}, \ref{geometric}. 

\begin{proof}
Note that it is enough to check the case when $n=2$. To simplify the notation, we denote $Y_1$ as $X$ and $Y_2$ as $Y$.
We will check reducedness at the level of the local ring at a closed point, and we only need to check the intersection $X\cap Y$.  Let $p \in X\cap Y$ a closed point, and $A$ be $\mathscr{O}_{Z,p}$. Since $X, Y$ are regularly embedded, the stalks of their ideal sheaves over the closed points can be generated by a regular sequence. 

Let $I_X:= \mathscr{I}_{X,p}$  be generated by a regular sequence $(a_1,a_2, \cdots, a_n)$ of length $n$. Then it is enough to assume $I_Y :=  \mathscr{I}_{Y,p}$ be generated by a regular sequence $ (a_1,a_2\cdots, a_e, b_{e+1}, \cdots b_m)$ of length $m$, where $a_i,b_i\in A$. 

To see why we can make such an assumption, note that $I_{Y}/I_{Y}^2$ and $I_{X}/I_{X}^2$ can be considered as free $A/I_{Y}$-module and $A/I_{X}$-module.  Since $X\cap Y$ is regularly embedded in both $X$ and $Y$, then we can consider both as free $A/I_{X\cap Y}$-submodules of $I_{X\cap Y}/I^2_{X\cap Y}$ after base change. In particular,  $(I_{Y}/I_{Y}^2 \otimes A/I_{X\cap Y})\cap (I_X/I_{X}^2 \otimes A/I_{X\cap Y})$ is also a free submodule of $I_{X\cap Y}/I^2_{X\cap Y}$ . Therefore one can choose a minimal set of generators for this free submodule, say $a_1,a_2,\cdots , a_e$. Then we can complete the minimal set of generators for both $I_{X}/I_{X}^2$ and $I_{Y}/I_{Y}^2$, say $a_1,\cdots, a_e, a_{e+1} \cdots a_e$ and $a_1,\cdots, a_e,b_{e+1},\cdots, b_n$ . These two sets form quasi-regular sequence. Note that by construction, $I_X, I_Y\subseteq rad(A) = m_p$ the maximal ideal associated to the closed point $p$. Then by \cite{mat06} Remark 15.1 ($\alpha$), we have these two sets also form a regular sequence. 

By transversality, we have $\mathscr{N}_{X}Z|_{X\cap Y} \oplus \mathscr{N}_{Y}Z|_{X\cap Y} \simeq \mathscr{N}_{X\cap Y}Z.$ This implies that $\mathscr{I}_X/\mathscr{I}^2_X |_{X\cap Y}\oplus \mathscr{I}_Y/\mathscr{I}^2_Y |_{X\cap Y}  \simeq \mathscr{I}_{X\cap Y}/\mathscr{I}^2_{X \cap Y} .$ Note that $(\mathscr{I}_X/\mathscr{I}^2_X)_{p}$ is isomorphic to $A/I_{X\cap Y}[a_1,a_2,\cdots ,a_n]$ as finitely generated $A/I_{X \cap Y}-$module. 
Then $I_{X\cap Y}/I_{X\cap Y}^2 \simeq A/I_{X\cap Y}[a_1,\cdots, a_n, b_{e+1}, \cdots, b_m]$ and  $rank_{A}(I_{X\cap Y}/I_{X\cap Y}^2)= m+n -e$. However, $ rank_{A}(I_{X}/I_{X}^2\oplus I_{Y}/I_{Y}^2) = m+n.$ Thus, we must have $e =0$.

Then $I_{X}\cdot I_{Y} = (a_ib_j)_{\substack{1\leq d \leq n\\ 1\leq s\leq m}}$ and since $X,Y$ are reduced, $I_X$ and $I_Y$ contain no such element $x$ such that $x^n\in I_X$ or $x^n\in I_X$ for any integer $n$. Then there is no such element $x$ such that $x^n\in I_X\cdot I_Y$  for any integer $n$. Since $Z$ is reduced, then $A$ is reduced. Hence $A/(I_X\cdot I_Y)$ is reduced.
\end{proof}
Applying \ref{product lemma}, this statement of \ref{many} specializes to a more geometric version:
\begin{cor} \label{geometric}
Let $Y_1, Y_2, \dots, Y_n$ be reduced regularly embedded subvarieties of the scheme $Z$, and $Y_1, Y_2, \dots, Y_n$ intersect transversely in $Z$. Then:
$$s(\bigcup_{i=1}^nY_i,Z)_d = \sum_{j=1}^{n}(-1)^{j+1}\sum_{\substack{i_1,\dots, i_j \in N\\ i_p \neq i_q}}\{\bigodot_{k=1}^js(Y_{i_k},Z)\}_d $$
\end{cor}

\section{The $\odot$ Notation and $Q$-polynomial}
Fulton's residual intersection \ref{fultonresidual} has a non-dimension specific expression\ref{alufresidual}, using the $\otimes$ operation. The goal for introducing $\odot$ is to give a non-dimension specific expressions for the formulas. This comparison is valid because, in this situation, we can think of the Segre class formally as polynomials whose grading is given by codimension.
Let $X$, $Y, Z,W$ be as in \ref{theorem1} , and let $E$ be the exceptional divisor of blowing up $Y$. The blowup is given by $\pi:Bl_{Y}Z \rightarrow Z$. Then one could try to apply Residual Intersection theorem \ref{alufresidual} $\mathscr{E}:= N_{Y}Z$ and $\mathscr{F}:= N_{X}Z$,  then we have:
$$s(W,Z) = \pi_*(s(E,\tilde{Z})+s(\tilde{X},\tilde{Z})^{\mathscr{O}(E)})$$
$$= s(Y,Z)+\pi_*\frac{[\tilde{X}]}{c(\mathscr{O}(E))c(\pi^*\mathscr{F}\otimes \mathscr{O}(E))}$$

Note that the above operator on $[X]$ is of the following form by the splitting principal and considering Chern roots:
$$\frac{1}{(1+E)(1+a_1+E)(1+a_2+E)\cdots(1+a_n+E)},$$
\begin{defn}
The $Q$-polynomial $Q(a_i,E)$ is defined as the polynomial such that the above expression can be expanded as follows:
$$\frac{1}{(1+a_1)(1+a_2)\cdots(1+a_n)}+ Q(a_1,a_2,\cdots, a_n,E)\cdot E,$$
where $Q$ is a polynomial symmetric with respect to $a_1,a_2,\cdots, a_n$.
\end{defn}
It is safe to let $Q$ take the parameters $e_1,e_2,\cdots e_n, E$,  where $e_i$ is the $n$-th elementary symmetric polynomial with variables $a_1,a_2,\cdots, a_n,$
$$ \frac{[\tilde{X}]}{c(\mathscr{O}(E))c(\pi^*\mathscr{F}\otimes \mathscr{O}(E))}= \frac{[\tilde{X}]}{c(\mathscr{\pi^*F})}+Q(c_i(\pi^*\mathscr{F}), E)\cdot E\cap [\tilde{X}].$$
Then,

$$s(W,Z)  =  s(Y,Z)+s(X,Z)+\pi_*(Q(c_i(\pi^*\mathscr{F}),E)\cdot E\cap [\tilde{X}])$$

By Fulton's Residual formula, we also have the expression in terms of $\odot$:
$$s(\pi^{-1}(W),\tilde{Z})= s(\tilde{X},\tilde{Z})+ s(E,\tilde{Z})-s(\mathscr{O}(E))\odot s(\pi^*\mathscr{F})\cap [E\cap \tilde{X}],$$
in $\odot$ notation.
Then, comparing the two terms supported strictly on $X\cap Y$:
$$\pi_*(Q(c_i(\pi^*\mathscr{F}), E)\cdot E\cap [\tilde{X}]) =-\pi_*(s(\mathscr{O}(E))\odot s(\pi^*\mathscr{F})\cap [E\cap \tilde{X}])$$
$$=-\pi_*\left(\sum_{k=0}^{\infty}\sum_{\substack{i+j=k\\i,j \geq 1}} \binom{k}{i}s_{j-rank(F)}(\mathscr{\pi^*F})\cdot (-1)^{i-1}E^{i-1}\cap [E\cap \tilde{X}]\right)$$
$$=-\sum_{k=0}^{\infty}\sum_{\substack{i+j=k\\i,j \geq 1}} \binom{k}{i}s_{j-rank(F)}(\mathscr{F})\cdot \pi_*\left((-1)^{i-1}E^{i-1}\cap [E\cap \tilde{X}]\right)$$
From the proof of \ref{theorem1}, we know that $\pi_*(E^i\cap \tilde{X}) = s(X\cap Y,X)_{dim(X)-i} = s_{i-(dimX-dimX\cap Y)}(N_{X\cap Y}X)\cap [X\cap Y].$
Hence, the above became:
$$=-\sum_{k=0}^{\infty}\sum_{\substack{i+j=k\\i,j \geq 1}} \binom{k}{i}s_{j-rank(F)}(\mathscr{F})\cdot s_{i-rank(N_{X\cap Y}X)}(N_{X\cap Y}X) \cap [X\cap Y]$$
$$=-s(\mathscr{F})\odot s(N_{X\cap Y}X)\cap[X\cap Y]$$
In the transversal case, we have $s(N_{X\cap Y}X) = s(\mathscr{E}|_{X\cap Y}) $.
This discussion is useful for the proof of \ref{theorem2}.
\vspace{5mm}
\section{Proof of Theorem \ref{theorem2}}
We will prove Theorem \ref{theorem2} by using a variation of Theorem 4.2 \cite{Alu09} which is \ref{chern class of blowup}  in this paper. The proof of said theorem will be presented in Section 6.
\vspace{5mm}
\begin{proof}
By residual intersection formula, we have,
$$s(W,Z) = \pi_*s(E\cup \tilde{X}, \tilde{Z}) = \pi_*(s(E,\tilde{Z}) + s(\tilde{X}, \tilde{Z})^{\mathscr{O}(E)})$$
We see that the only thing we need to compute is $\pi_*(s(\tilde{X}, \tilde{Z})^{\mathscr{O}(E)}).$
Note that $s(\tilde{X},\tilde{Z}) = \frac{1}{c(N_{\tilde{X}}\tilde{Z})}\cap [\tilde{X}]$
%By the theorem in Chern Class of Blowup, 
%Consider the commutative diagram:
%\begin{tikzcd}
 %   E \arrow[r, "j"]\arrow[d,"g"] & \tilde{Z}\arrow[d, "\pi"] \\
  %  Y \arrow[r, "i"]& Z
%\end{tikzcd}

% and the diagram of restricting to $X$:
% \begin{tikzcd}
  %$$  \tilde{X\cap Y} \arrow[r, "j'"]\arrow[d,"g'"] & \tilde{X}\arrow[d, "\pi'"] \\
  %  X\cap Y \arrow[r, "i'"]& X
%\end{tikzcd}
%we have the formula $c(T_{\tilde{Z}})\cap [\tilde{Z}] = \pi^*c(T_Z)\cap [Z]+ j_*(g^*c(T_Y)\cdot \alpha) $, where $$\alpha = \frac{1}{c_1(\mathscr{O}_N(1))}\Bigl(\sum_{i=0}^d g^*c_{d-i}(N)- (1-c_1(\mathscr{O}_N(1))\sum_{i=0}^d(1+c_1(\mathscr{O}_N(1)))^ig^*c_{d-i}(N)\Bigl),$$
%where $N = N_{Y/Z}.$
%We also have the same statement for $c(T_{X})$ as well.

By \ref{chern class of blowup}, using the notation for $\Gamma'$, there is:
$$c(N_{\tilde{X}}\tilde{Z})-c(\pi^*N_XZ)= \Gamma' = c(\pi'^*N_{X'}A')\zeta(E',\pi'^*c_i(N_{A'}Z'))\cdot E',$$
where $X', A', Z'$ is defined in the following diagram, 
\[
\begin{tikzcd}
    X\cap Y \arrow[r,"i"]\arrow[d,"j"]& X' = P(N_{X\cap Y}X \oplus 1)\arrow[d]\\
   Y' =  p(N_{X\cap Y} Y \oplus 1) \arrow[r]& A' = P(N_{X\cap Y}X\oplus N_{X\cap Y}Y \oplus 1)\arrow[dr]\\
     & &  Z' = P(N_{X\cap Y}Z \oplus 1)
\end{tikzcd}
\]
Furthermore, $\pi': Bl_{Y'}Z' \rightarrow Z'$ is the blowup map, and $E'$ is the exceptional divisor of the blowup.

Then, we observe that, 
$$s(\tilde{X}, \tilde{Z})^{\mathscr{O}(E)} = \left(\frac{[\tilde{X}]}{c(N_{\tilde{X}}\tilde{Z})}\right)^{\mathscr{O}(E)}= \left(\frac{[\tilde{X}]}{c(\pi^*N_{X}Z)+ \Gamma'}\right)^{\mathscr{O}(E)}$$
Notice that the since $\Gamma$ is supported on $E\cap \tilde{X}$. In the expansion, the only term whose support contains  $\tilde{E}\cap X$ is $\frac{[\tilde{X}]}{c(\pi^*N_XZ)}$.  Denote $\mathbb{R}$ as the residual term that is supported on $\tilde{X}\cap E$ . $\mathbb{R}$ has a factor of $E$ and therefore is supported on $\tilde{X}\cap E.$ More specifically,
$$\mathbb{R}=\pi_*\left(s(\tilde{X}, \tilde{Z})^{\mathscr{O}(E)} -\frac{[\tilde{X}]}{c(\pi^*N_XZ)}\right) = \pi_*\left(s(\tilde{X}, \tilde{Z})^{\mathscr{O}(E)}\right) -s(X,Z)$$

To determine $\mathbb{R}$, it is sufficient to restrict to the support $\tilde{X}\cap E.$ Let $i: \tilde{X}\cap E \hookrightarrow \tilde{X}$.
Then,  
$$\mathbb{R} =\pi_*\left(i^*s(\tilde{X}, \tilde{Z})^{\mathscr{O}(E)} -i^*\left(\frac{[\tilde{X}]}{c(\pi^*N_XZ)}\right)\right)$$
$$=\pi_*\left(i^*s(\tilde{X}, \tilde{Z})^{\mathscr{O}(E)} -\frac{[\tilde{X}\cap E]}{c(\pi^*N_XZ)}\right) $$
Note that $c(\pi^*N_{X}Z|_{\tilde{X}\cap E}) = c(\pi'^*N_{X'}Z'|_{\tilde{X}\cap E})$, then,
$$c(N_{\tilde{X}}\tilde{Z}|_{\tilde{X}\cap E}) = c(N_{\tilde{X'}}\tilde{Z}'|_{\tilde{X}\cap E})= \Gamma'|_{\tilde{X}\cap E}+c(\pi'^*N_{X'}Z'|_{\tilde{X}\cap E})$$

By the definition of $\Gamma'$, this is simply:
$$c(\pi^*N_{X'}A'|_{\tilde{X}\cap E})c((\pi^*N_{A'}Z'\otimes \mathscr{O}(-E'))|_{\tilde{X}\cap E})$$

Furthermore, we have $N_{X'}A'|_{X\cap Y} = \frac{N_{X\cap Y}X\oplus N_{X\cap Y} Y}{N_{X\cap Y}X}= N_{X\cap Y}Y$,  $N_{A'}Z'|_{X\cap Y} = \frac{N_{X\cap Y}Z}{N_{X\cap Y}X\oplus N_{X\cap Y} Y}$
Then, putting this together, we have:

$$c(N_{\tilde{X}}\tilde{Z}|_{\tilde{X}\cap E})=c(\pi^{*}N_{X\cap Y}Y)c((\pi^{*}\frac{N_{X\cap Y }Z}{N_{X\cap Y}X \oplus N_{X\cap Y}Y}\otimes \mathscr{O}(-E))|_{\tilde{X}\cap E})$$

Assume that all classes are restricted to the support $\tilde{X}\cap E$. To simplify the computation, we will omit indicating the restriction of the normal bundle within the Chern class. 
\begin{align*}
\pi_*\left(i^*s(\tilde{X}, \tilde{Z})^{\mathscr{O}(E)}\right)&=\frac{c(N_{X\cap Y}X) c(N_{X\cap Y}Y)}{c(N_{X\cap Y}Z)}\pi_*\left(\frac{1}{(1+E)c\left(\pi^{*}N_{X\cap Y}Y\otimes \mathscr{O}(E)\right)}\cap [\tilde{X}]\right)\\
& =\frac{c(N_{X\cap Y}X) c(N_{X\cap Y}Y)}{c(N_{X\cap Y}Z)}\pi_*\left (\frac{[\tilde{X}]}{c(\pi^{*}N_{X\cap Y}Y)}+  Q(c_i(\pi^{*}N_{X\cap Y}Y), E)\cdot E\cap [\tilde{X}]\right),
\end{align*}

where $Q$ is the polynomial described in Section 3. Note that $\frac{c(N_{X\cap Y}X) c(N_{X\cap Y}Y)}{c(N_{X\cap Y}Z)}\cdot\frac{1}{c(N_{X\cap Y}Y)} =\frac{1}{c(N_{X}Z|_{X\cap Y})}$ .
Since $E, \tilde{X}$ are residual to each other, then we have $[\tilde{X}\cap E] = [\tilde{X}]\cdot E$. By universal property, we also have that $\pi_*([\tilde{X}\cap E]) = [X\cap Y]$.  Then, by previous discussion of the $Q$ polynomial in Section 3, 
$$\mathbb{R} = \frac{c(N_{X\cap Y}X) c(N_{X\cap Y}Y)}{c(N_{X\cap Y}Z)}\pi_*(Q(c_i(\pi^{*}N_{X\cap Y}Y), E)\cap [E\cap \tilde{X}])$$$$ = -\frac{c(N_{X\cap Y}X) c(N_{X\cap Y}Y)}{c(N_{X\cap Y}Z)}\left(s(N_{X\cap Y}X) \odot s(N_{X\cap Y}Y)\right)\cap [X\cap Y]$$

Putting it all together, we have the formula:
$$s(W,Z) = s(X,Z)+s(Y,Z) -\frac{c(N_{X\cap Y}X) c(N_{X\cap Y}Y)}{c(N_{X\cap Y}Z)}\left(s(N_{X\cap Y}X) \odot s(N_{X\cap Y}Y)\right)\cap [X\cap Y]. $$

\end{proof}
\begin{remark}
The expression for $s(W,Z)_d$ is the following:
$$s(X,Z)_d+s(Y,Z)_d-$$
\footnotesize{$$\sum_{\substack{i+j = dim Z - d \\i,j \geq 0}} \left\{\frac{c(N_{X\cap Y}X)c(N_{X\cap Y}Y)}{c(N_{X\cap Y} Z)}\right\}^i\sum_{\substack{k+l = j-h\\ k,l\geq 0}}\binom{j-h}{k,l}s_{k-(dimY-dimX\cap Y)}(N_{X\cap Y}Y)s_{l-(dimY-dimX\cap Y)}(N_{X\cap Y}X)\cap [X\cap Y],$$}
where $h = dimZ+dimX\cap Y - dimX -dimY$.

\end{remark}
\begin{exmp}
Let $\mathbb{P}^m, \mathbb{P}^n\hookrightarrow \mathbb{P}^N$ , and suppose $\mathbb{P}^n\cap \mathbb{P}^m \simeq \mathbb{P}^l$. Consider the scheme $W$ supported on $\mathbb{P}^n\cup \mathbb{P}^m$ that is reduced. Then one could check by local computation that blowing up $\mathbb{P}^n$,  $\tilde{\mathbb{P}}^m $ is residual to $E$ in $\tilde{W}$. 
Then we obtained that, the coefficient of $H^{N-d}$ in $s(W,\mathbb{P}^n)$ is:
$$(-1)^{m-d}\binom{N-d-1}{m-1}+(-1)^{n-d}\binom{N-d-1}{n-1}+$$
\footnotesize{$$\sum_{\substack{i,j\geq 0\\ i+j = N-d}}(-1)^i\binom{N-m-n+l+i-1}{N-m-n+l-1}\sum_{u+v = j-(N-m-n+l)}(-1)^{j-l}\binom{j-(N-m-n+l)}{u}\binom{u-1}{m-l-1}\binom{v-1}{n-l-1}$$}
\end{exmp}
\section{Appendix: Chern Class of Blowup}

\hspace{8mm} In \cite{Alu09} section 4.4, the author uses the technique of deformation to the normal cone (bundle) to compute the Chern class of the normal bundle to the proper transform. We will recount the argument here and use a variation of the result to express the Chern class of blowup explicitly in terms of the normal bundles before blowing up. 

Considering the following diagram (\textdaggerdbl):
\[
\begin{tikzcd}
X\cap Y \arrow[hookrightarrow]{r} \arrow[hookrightarrow]{d} & Y \arrow[hookrightarrow,"j"]{d}\\
X \arrow[hookrightarrow]{r}& Z \\
\end{tikzcd}
\]
Let $\pi: Bl_{Y}Z \rightarrow Z$ be the blowup, and $E$ is the exceptional divisor. The goal for this section is to produce an explicit formula to compute the Chern class of the normal bundle $N_{\tilde{X}}\tilde{Z}$ in terms of the pullback of the bundle in the original diagram. 

We consider the following simpler case first:
\begin{lem}
Let $X,Y$ be a regularly embedded subscheme of $Z$. $X\cap Y$ regularly embedded in $X$ and $Y$.
Suppose there is a regularly embedded subscheme $A$ of $Z$, where $X$ intersects $Y$ transversely. 
Then,
$$c(N_{\tilde{X}}\tilde{Z}) = c(\pi^*N_{X}A)c\left((\pi^*N_{A}Z\otimes \mathscr{O}(-E))|_{\tilde{X}}\right)$$
\end{lem}
\begin{proof}
Consider the following diagram:
\[
\begin{tikzcd}
X\cap Y \arrow[hookrightarrow]{r} \arrow[hookrightarrow]{d} & Y \arrow[hookrightarrow]{d}\\
X \arrow[hookrightarrow]{r}& A \arrow[hookrightarrow]{rd} \\
 & & Z
\end{tikzcd}
\]
Consider blowing up $Y$, with $\pi: Bl_YZ\rightarrow Z$  and $E$ as the exceptional divisor, since $X,Y$ are also regularly embedded in $A$, by lemma 2.3(a), we have $N_{\tilde{X}}\tilde{A} = \pi_*N_{X}A.$ 
By \cite{Ful84} Appendix B 6.10, we also have $N_{\tilde{A}}\tilde{Z} = \pi^*N_{A}Z \otimes \mathscr{O}(-E).$
Since we also have the exact sequence:
$$0\rightarrow N_{\tilde{X}}\tilde{A} \rightarrow N_{\tilde{X}}\tilde{Z} \rightarrow N_{\tilde{A}}\tilde{Z}|_{\tilde{X}}\rightarrow 0$$
Then we have $$c(N_{\tilde{X}}\tilde{Z}) = c(\pi^*N_{X}A)c\left((\pi^*N_{A}Z\otimes \mathscr{O}(-E))|_{\tilde{X}}\right)$$
\end{proof}
This formula is true even without a priori knowing the existence of $A$. In fact it is enough to replace $X,Y,Z$ with the projective bundle $P(N_{X\cap Y}X\oplus 1) , P(N_{X\cap Y}Y \oplus 1),P(N_{X\cap Y}Z\oplus 1) .$ In this case, we recover previous situation by considering $P(N_{X\cap Y}X  \oplus N_{X\cap Y}Y \oplus 1) $, and note that $P(N_{X\cap Y}X \oplus 1)$ and $P(N_{X\cap Y}X \oplus 1) $ intersect transversely in $P(N_{X\cap Y}X  \oplus N_{X\cap Y}Y\oplus 1) \subseteq P(N_{X\cap Y}Z \oplus 1)$. We have the following diagram (\textdaggerdbl'):
\[
\begin{tikzcd}
    X\cap Y \arrow[r,"i"]\arrow[d,"j"]& X' = P(N_{X\cap Y}X \oplus 1)\arrow[d]\\
   Y' =  P(N_{X\cap Y} Y \oplus 1) \arrow[r]& A' = P(N_{X\cap Y}X\oplus N_{X\cap Y}Y \oplus 1)\arrow[dr]\\
     & &  Z' = P(N_{X\cap Y}Z \oplus 1)
\end{tikzcd}
\]
Let $\pi':Bl_{Y'}Z'\rightarrow Z'$ denotes the blowup of $Z'$ at $Y'$, and let $E'$ denote the exceptional divisor.
Then the previous Lemma tells us:
$$c(N_{\tilde{X}'}\tilde{Z}') = c(\pi'^*N_{X'}A')c((\pi'^*N_{A'}Z'\otimes \mathscr{O}(-E'))|_{\tilde{X}'})$$
Then we define the operator $\Gamma'$ in the following way:
Note that 
$$c((\pi'^*N_{A'}Z'\otimes \mathscr{O}(-E'))|_{\tilde{X}'}) = c(\pi'^*N_{A'}Z')+ (\zeta(E',\pi'^*c_i(N_{A'}Z'))\cdot E')|_{\tilde{X}'},$$
where $\zeta$ is some formal polynomial with variables $E'$ and $\pi'^*c_i(N_{A'}Z')$. There is a factor of $E'$ in the second summand. Therefore it is supported on $E'\cap \tilde{X}$, and we can drop the restriction.
\begin{defn}
Define $\Gamma' := c(\pi'^*N_{X'}A')\zeta(E',\pi'^*c_i(N_{A'}Z'))\cdot E'.$ Moreover, $\Gamma'$ has the property:
$$\Gamma'= c(N_{\tilde{X}'}\tilde{Z}') - c(\pi'^*N_{X'}Z') $$
\end{defn}
\vspace{5mm}
Similarly, we can also define $\Gamma :=c(N_{\tilde{X}}\tilde{Z}) - c(\pi^* N_{X}Z)$ for the standard case given by the diagram (\textdaggerdbl).
Note that we have an explicit expression for $\Gamma'$ but not for $\Gamma$. We will use the technique of deformation to the normal cone to establish equality.
Deformation to the normal cone is introduced in \cite{Ful84} chapter 5. It is shown that there exists a scheme $M_Z: = Bl_{X\cap Y\times \{\infty\}}Z \times \mathbb{P}^1$, such that consider the projection $\sigma: M_Z \rightarrow \mathbb{P}^1$, then $\sigma^{-1}(\infty) =  P(N_{X \cap Y}Z\oplus 1) \cup Bl_{X\cap Y} Z$ set-theoretically, $\sigma^{-1}(0) = Z$.
\begin{prop}\label{chern class of blowup}
Consider the situation in diagram (\textdaggerdbl), and with the above definition, 
\begin{enumerate}[label=(\alph*)]
\item$\tilde{X}'\cap E' = \tilde{X}\cap E$ 
\item  $\Gamma' = \Gamma$
\end{enumerate}
\end{prop}
\vspace{5mm}
This proposition extends Theorem 4.2 in \cite{Alu09} which is proven using the technique of deformation to the normal cone in \cite{Alu09} section 4.6. We will recount the argument here and explicitly determine the bundles $\mathbb{N}$ and $\mathbb{C}$ mentioned in the paper:
\begin{proof}
Note that $\Gamma$ is supported on $\tilde{X}\cap E$, and $\Gamma'$ is supported on $\tilde{X}'\cap E'$, so (2) only makes sense if (1) holds. We will prove (1) first.
\newline Remember that the intersection $X'\cap Y'$ is with respect to the canonical embedding $N_{X\cap Y}X,N_{X\cap Y}Y  \hookrightarrow N_{X\cap Y}X\oplus N_{X\cap Y}Y $. Then the intersection would simply be the base of the normal bundles, namely $X'\cap Y' = X\cap Y$.

Now we claim, $\tilde{X}\cap E =  \tilde{X}' \cap E'$ 
Since $\tilde{X}'\cap E'$ is the exceptional divisor of the blowup $Bl_{X'\cap Y'}X'$, $\tilde{X}'\cap E'= P(N_{X' \cap Y'}X')= P(N_{X\cap Y}X).$ Note that $P(N_{X\cap Y}X)$ is also the exceptional divisor of $Bl_{X\cap Y}X$, then $\tilde{X}'\cap E' = \tilde{X}\cap E.$

To prove (2), we will use the technique of deformation to the normal cone. 
Let $M_X, M_Y$ and $M_Y$ be the deformation to the normal cones to $X\cap Y$. Then consider the blowup: $\rho:\tilde{M}_Z := Bl_{M_Y}M_Z\rightarrow M_Z$, this induces $\bar{\pi}: \tilde{M}_X\rightarrow M_X.$ Note that $M_X\cap M_Y = (X\cap Y)\times \mathbb{P}^1$. Therefore, $\tilde{M}_X = Bl_{(X\cap Y)\times\mathbb{P}^1} M_X$. We will use $\bar{\pi}$ to connect the two projection maps of blowup $\pi:Bl_{Y}Z\rightarrow Z$, $\pi': Bl_{Y'}Z'\rightarrow Z'$ that we have introduced before.
Consider the diagram for $\tilde{M}_X$:
\[
\begin{tikzcd}
 &&\tilde{M}_Z:= Bl_{X\cap Y \times \mathbb{P}^1}M_Z\arrow[ld,"\bar{\pi}"]\arrow[rd, "\nu"]\arrow[drr, bend left ,"p"]\\&M_Z\arrow[rd] & &\tilde{Z}\times \mathbb{P}^1\arrow[ld]\arrow[r]&\tilde{Z}\\
& &X\times\mathbb{P}^1& &
\end{tikzcd}
\]
where $z$ is the inclusion map into the fiber of $\tilde{M}_Z$ at $0$, $\bar{\pi}$ is the projection map of blowup $M_Z$ at the isomorphic copy $(X\cap Y)\times \mathbb{P}^1$,  and $\nu$ is the projection map of blowing up $\tilde{Z}\times \mathbb{P}^1$ at $\widetilde{X\cap Y}\times\{\infty\}$. Note that $\tilde{M}_Z$ is the deformation to the normal cone (bundle) $N_{\widetilde{X\cap Y}}\tilde{Z}$. By unviersal propety of blowup $\tilde{M}_Z$ is well-defined and makes the diagram commutes. Furthermore note that $\bar{\pi}$ restricts to the fiber over any points on $\mathbb{P}^1$ that is not $\infty$ is $\pi$, and the restriction at $\infty$ is $\pi'$. The map $p$ is defined to be the composition of $\tilde{M}_Z\rightarrow \tilde{Z}\times \mathbb{P}^1\rightarrow \tilde{Z}.$

%By previous discussion on $\Gamma$, since there is a factor on the exceptional divisor, then 
%$c(N_{\tilde{X}}\tilde{Z}) - c(\pi^*N_{X}Z) = c(N_{\tilde{X}}\tilde{Z}|_{\tilde{X}\cap E}) - c(\pi^*N_{X}Z|_{\tilde{X}\cap E})) $
Let $V$ be any subvariety of $\tilde{X}$, then it is enough to prove that $\Gamma\cap [V] = \Gamma'\cap [V].$
By projection formula, $$(c(N_{\tilde{X}}\tilde{Z}) - c(\pi^*N_{X}Z))\cap [V] = (c(N_{\tilde{M}_X}\tilde{M}_Z)- c(\bar{\pi}^*N_{M_X}M_Z))\cap[V \times \{0\}] $$
By rational equivalence, $[V \times \{0\}] = [P(N_{\widetilde{X\cap Y}\cap V}V\oplus 1)] + [Bl_{\widetilde{X\cap Y}\cap V}V].$
Then,
 $$(c(N_{\tilde{X}}\tilde{Z}) - c(\pi^*N_{X}Z))\cap [V] = $$
 $$p_*\left((c(N_{\tilde{M}_X}\tilde{M}_Z)- c(\bar{\pi}^*N_{M_X}M_Z))\cap([P(N_{\widetilde{X\cap Y}\cap V}V\oplus 1)] + [Bl_{\widetilde{X\cap Y}\cap V}V])\right) $$
But $ Bl_{\widetilde{X\cap Y}\cap V} V $ is disjoint from the center of the blowup $\bar{\pi}$ .  Therefore,
 $$(c(N_{\tilde{M}_X}\tilde{M}_Z)- c({\pi}^*N_{M_X}M_Z))\cap [Bl_{\widetilde{X\cap Y}\cap V}V] = 0$$
and,
 $$ (c(N_{\tilde{X}}\tilde{Z}) - c(\pi^*N_{X}Z))\cap [V] = $$
\begin{align}
p_*\left((c(N_{\tilde{M}_X}\tilde{M}_Z)- c(\bar{\pi}^*N_{M_X}M_Z))\cap[P(N_{\widetilde{X\cap Y}\cap V}V\oplus 1)]\right)
\end{align}
Now we wish to express the fiber of the normal bundle at infinity in terms of normal bundle as in diagram (\textdaggerdbl'). Note that by the restriction condition, it is enough to consider the $N_{\tilde{X}\cap E}\tilde{X}$ for $\tilde{M}_X$, and $\tilde{Z}'$ for $\tilde{M}_Z.$
This is because over the fiber at $\infty$, $\tilde{X}\cap E \subseteq N_{\tilde{X}\cap E} \tilde{X}$, and because $X\cap Y = X'\cap Y',$
then $\tilde{X}' = Bl_{X'\cap Y'} X' = Bl_{X\cap Y}N_{X\cap Y}X= N_{\tilde{X}\cap E} \tilde{X} $.
\newline

Therefore, we obtain the following:
$$N_{\tilde{M}_X}\tilde{M}_Z = p^*(N_{\tilde{X}'}\tilde{Z}' )$$
 and similarly,
 $$\bar{\pi}^*N_{M_X}M_Z  = p^*\pi'^*(N_{X'}Z')$$
Now we have reduced the situation to the case of diagram (\textdaggerdbl'), and we have $c(N_{\tilde{X}'}\tilde{Z}') - c(\pi'^*N_{X'}Z') = \Gamma',$ .and therefore Hence, $c(N_{\tilde{M}_X}\tilde{M}_Z)- c(\bar{\pi}^*N_{M_X}M_Z) = p^*\Gamma'$. Then, 
$$(c(N_{\tilde{X}}\tilde{Z}) - c(\pi^*N_{X}Z))\cap [V] =p_*\left(p^*\Gamma'\cap[P(N_{\widetilde{X\cap Y}\cap V}V\oplus 1)]\right)$$

Again we have $p^*\Gamma \cap [Bl_{\widetilde{X\cap Y}\cap V}V] = 0$ by the disjointness of center of blowup. Then one fill in the missing class and transfer the class back to fiber at $0$ through rational equivalence, namely,
$$\Gamma\cap [V] = p_*(p^*\Gamma'\cap ([P(N_{\widetilde{X\cap Y}\cap V}V\oplus 1)]+[Bl_{\widetilde{X\cap Y}\cap V}V])) = p_*(p^*\Gamma'\cap [V\times \{0\}]).$$
Then by projection formula $\Gamma\cap [V] = \Gamma'\cap [V].$

\end{proof}

\bibliographystyle{alpha}

\bibliography{references.bib}

% \begin{thebibliography}
% \bibitem{Alu94}[Alu94]
% Paolo Aluffi, \textit{MacPherson's and Fulton's Chern Classes of Hypersurfaces},
% International Mathematics Research Notices, (1994).
% \bibitem{Alu08}[Alu08]
% Paolo Aluffi, \textit{Chern classes of blow-ups}, Mathematical Proceedings of the Cambridge Philosophical Society, (2008).
% \bibitem{Alu15}[Alu15]
% Paolo Aluffi, Degrees of Projections of Rank Loci, Experimental Mathematics, 24(4), Page 469-488, (2015).
% \bibitem{Ful84}[Ful84]
% William Fulton, \textit{Intersection Theory},
% Springer-Verlag, Berlin, (1984).
% \bibitem{Ges13}[Ges13]
% Gesmundo, F., Hauenstein, J.D., Ikenmeyer, C. et al, \textit{Complexity of Linear Circuits and Geometry}, Found Comput Math, 16, Page 599-635, (2013).
% \bibitem{Hau03}[Hau03]
% Herwig Hauser, \textit{The Hironaka Theorem on Resolution of Singularities}, Bulletin of the American Mathematical Society, (2003).

% \bibitem{Kiem10}[Kiem10]
% Young-Hoon Kiem, \textit{Moduli Spaces of Weighted Pointed Stable Rational Curves Via GIT}, (2010).

% \bibitem{Li09}[Li09]
% Li Li, \textit{Wonderful Compactification of an Arrangement of Subvarieties},
% Michigan mathematical journal, 58(2), Page 535-563, (2009).

% \end{thebibliography}

\end{document}